\documentclass[12pt,normalheadings]{scrartcl}  % `article' with A4 paper size option
\usepackage{amsmath}
\usepackage{amsthm}                 % theorem extensions
\usepackage{amssymb}
\usepackage{graphicx}
\usepackage{psfrag}
\usepackage{caption}

\theoremstyle{plain}
\newtheorem{theorem}{Theorem}

\newtheorem{lemma}[theorem]{Lemma}

\newtheorem{proposition}[theorem]{Proposition}

\theoremstyle{definition}
\newtheorem{remark}[theorem]{Remark}
\newtheorem{example}[theorem]{Example}

\newcommand{\CC}{{\mathbb{C}}}
\newcommand{\HH}{{\mathbb{H}}}

\newcommand{\QQ}{{\mathbb{Q}}}

\newcommand{\RR}{{\mathbb{R}}}
\newcommand{\ZZ}{{\mathbb{Z}}}

\newcommand{\calO}{{\cal O}}

\newcommand{\calI}{{\cal I}}

\newcommand{\calE}{{\cal E}}

\newcommand{\bE}{{\calE}}
\newcommand{\e}{{E}}
\newcommand{\be}[2]{{E_{#1}^{#2}}}
\newcommand{\E}[2]{{E_{#1}^{#2}}}
\newcommand{\F}[2]{{F_{#1}^{#2}}}
\newcommand{\rad}{\text{rad}}

\DeclareMathOperator{\Div}{\mathrm{Div}}
\DeclareMathOperator{\Pic}{\mathrm{Pic}}
\DeclareMathOperator{\Hom}{\mathrm{Hom}}

\DeclareMathOperator{\Ext}{\mathrm{Ext}}
\DeclareMathOperator{\Coh}{\mathrm{Coh}}

\newcommand{\TTT}{\mathsf{T}\!}   % special case for the \TTT_X notation
\newcommand{\MMM}{\mathsf{M}}
\newcommand{\DDD}{\mathcal{D}}
\newcommand{\chern}{\mathrm{ch}}
\newcommand{\Dsgr}{\DDD_{\mathrm{sg}}^{\mathrm{gr}}}

\newcommand{\pseriesl}{[\![}
\newcommand{\pseriesr}{]\!]}

\newcommand{\isom}{ \text{{\hspace{0.48em}\raisebox{0.8ex}{${\scriptscriptstyle\sim}$}}} \hspace{-0.65em}{\rightarrow}\hspace{0.3em}} % Gewalt, ohne Trennung

\newcommand{\pares}{Y} % partial resolution in the Fuchs case

\begin{document}

\begin{center}
{\Large\textsf{\textbf{Poincar\'e series and Coxeter functors for \\ Fuchsian singularities}}}

\vspace{3ex}

Wolfgang Ebeling and David Ploog\footnote{
AMS Math.\ Subject Classification: 13D40, 18E30, 32S25.
\newline
The second author likes to thank the DFG whose grant allows him to enjoy the hospitality of the University of Toronto.
}
\end{center}

%\marginpar{\today}

\begin{quotation}
\noindent
\small
\textsc{Abstract}
We consider Fuchsian singularities of arbitrary genus and prove, in a conceptual manner, a formula for their Poincar\'e series. This uses Coxeter elements involving Eichler-Siegel transformations. We give geometrical interpretations for the lattices and isometries involved, lifting them to triangulated categories.
\end{quotation}

%%%%%%%%%%%%%%%%%%%%%%
\section*{Introduction}
A Fuchsian singularity is the affine surface singularity obtained from the cotangent bundle of the upper half plane by taking the quotient by a Fuchsian group of the first kind and collapsing the zero section. In particular, it has a good $\CC^\ast$-action. The surface can be compactified in a natural manner, leading to additional cyclic quotient singularities of type $A_\mu$ on the boundary. After resolving the singularities on the bundary, one gets a star-shaped configuration of rational $(-2)$-curves with a central curve of genus $g$ and self-intersection number $2g-2$. 

In the case $g=0$, the dual graph of this configuration determines a Coxeter element. It was shown in  \cite{EP} that the Poincar\'e series of the graded coordinate ring of the singularity is the quotient of the characteristic polynomials of two suitable extensions of this Coxeter element. 

Here we treat Fuchsian singularities of arbitrary genus $g$. If $g>0$, there is no longer a reflection defined by the homology class of the central curve. Therefore, one has to modify the definition of the Coxeter element. 
We replace the product of two reflections by an Eichler-Siegel transformation. With this change, we prove a result analogous to the one stated above, along the lines of \cite{EP}. We also give a geometrical and categorical interpretation of the Coxeter elements, thereby explaining where the Eichler-Siegel transformation comes from and why the methods applied before have to break down.

%%%%%%%%%%%%%%%%%%%%%%%%%%%%%%%%%%%%%%
\section{Eichler-Siegel transformations} \label{Euler-Siegel}
We first recall the definition of the Eichler-Siegel transformations.
Let $(V, \langle \ , \ \rangle)$ be an even integral lattice and denote by $\text{O}(V)$ the group of isometries of this lattice. Define a map
$ \Psi \colon V \otimes V  \to  \text{End}(V)$,
   %\qquad
   $\sum\nolimits_i u_i \otimes a_i \mapsto \text{id} - 
              \sum\nolimits_i \langle \cdot,a_i \rangle u_i.$% \]

Let $a \in V$ be arbitrary and $u \in V$ isotropic and orthogonal to $a$, i.e.\
 $\langle u,u\rangle=\langle a,u\rangle = 0$.
The \emph{Eichler-Siegel transformation} corresponding to $u$ and $a$ is defined as
\[\psi_{u,a} := \Psi((\tfrac{1}{2}\langle a, a \rangle u-a) \otimes u ) \
               \Psi(u \otimes a). \]
It is easily checked that $\psi_{u,a}$ is an isometry, using the formula
\begin{align*}
\psi_{u,a}(x) %&= x - \langle x,u \rangle \left( \tfrac{1}{2}\langle a,a \rangle u - a \right) 
              %     - \langle x,a \rangle u \\
              &= x + \langle x,u \rangle a -  \langle x,a \rangle u  
                   - \tfrac{1}{2}\langle a,a \rangle \langle x,u \rangle u.
\end{align*}

\begin{example}
Let $a \in V$ be a root, i.e.\ $\langle a, a \rangle =-2$. Then $\Psi(-a \otimes a)=s_a$  is the reflection corresponding to $a$, given by $s_a(v) = v + \langle v,a\rangle a$ for all $v \in V$. Furthermore, for $u\in V$ with $\langle u,u\rangle=\langle a,u\rangle = 0$, one easily sees
 $\psi_{u,a} = \Psi((-u-a) \otimes u) \ \Psi(u \otimes a) = s_a s_{a-u}$.
\end{example}

\begin{example} \label{ex:U}
Let $V_-$ be any even lattice. Denote by $U$ the unimodular hyperbolic
plane with a symplectic basis $u, w$  and the symmetric bilinear form
 $\langle u,w\rangle=1$
and $\langle u,u\rangle = \langle w,w\rangle = 0$.
Let $V_+ := V_-\oplus U$ be the orthogonal direct sum and define a group
homomorphism
 $m \colon V_- \to \text{O}(V_+)$, $a \mapsto m_a := \psi_{u,a}$.

The isometry $m_a$ is given by $m_a(u)=u$, $m_a(v)=v-\langle v,a\rangle u$ 
for any $v\in V_-$ and $m_a(w)=w+a-\tfrac{1}{2}\langle a,a\rangle u$.
This example appears in \cite[I, \S 3]{Ei}. (M.~Eichler notes that these 
automorphisms first occurred in a paper of C.~L.~Siegel.)
\end{example}

%%%%%%%%%%%%%%%%%%%%%%%
\section{The result: Poincar\'e series of Fuchsian singularities} \label{sec:result}
Let $(X,x)$ be a normal surface singularity with a good $\CC^\ast$-action. 
This means that $X=\text{Spec}(A)$ is a normal two-dimensional affine 
algebraic variety over $\CC$ which is smooth outside its \emph{vertex} $x$. 
Its coordinate ring $A$ has the structure of a graded $\CC$-algebra 
$A = \bigoplus_{k=0}^\infty A_k$, $A_0=\CC$, and $x$ is defined 
by the maximal ideal $\mathfrak{m}= \bigoplus_{k=1}^\infty A_k$. 

A natural compactification of $X$ is given by $\overline{X}:= {\rm Proj}(A[t])$, 
where $t$ has degree 1 for the grading of $A[t]$ (see \cite{Pinkham77a}). 
This is a normal projective surface with $\CC^*$-action, and $\overline{X}$ 
may acquire additional singularities on the boundary 
$\overline{X}_\infty:=\overline{X}\setminus X={\rm Proj}(A)$ which itself is
a smooth, projective curve.

A normal surface singularity $(X,x)$ with good $\CC^\ast$-action  is called 
\emph{Fuchsian} if the canonical sheaf of $\overline{X}$ is trivial. In this 
case, the singularities on the boundary are all of type $A_\mu$. The genus of 
the Fuchsian singularity is defined as the genus $g=g(\overline{X}_\infty)$ of 
the boundary.

Let $(X,x)$ be a Fuchsian singularity of genus $g$. According to \cite{Dolgachev75}
and \cite[(1.2) Proposition]{Looijenga84}, there exists a finitely generated
cocompact Fuchsian group of the first kind $\Gamma \subset PSL(2,\RR)$ 
which acts properly discontinuously on the upper half plane $\HH$ such that the
action of $\Gamma$ lifts to the cotangent bundle $T_{\HH}^{-1}$ and 
$A_k = H^0(\HH, T_\HH^{-k})^\Gamma$. The quotient $Z=\HH/\Gamma$ is a compact 
Riemann surface. Let $D_0$ be a canonical divisor of $Z$. By 
\cite[Theorem~5.1]{Pinkham77a} there exist points $p_1, \dotsc , p_r \in Z$ 
and integers $\alpha_i>1$ for $i=1, \dotsc, r$ such that
\[ A_k =L (D^{(k)}), \quad 
   D^{(k)} := kD_0 + \sum_{i=1}^r \big[ k \tfrac{\alpha_i - 1}{\alpha_i} \big] p_i
  \quad\text{ for } k\geq0 . \]
Here, $[x]$ denotes the largest integer $\leq x$, and $L(D):=H^0(Z,\calO_Z(D))$ 
for a divisor $D$ on $Z$ denotes the linear space of meromorphic functions $f$ 
on $Z$ such that $(f) \geq -D$. The genus $g$ of $Z$ coincides with the genus 
of $\overline{X}_\infty$. The degree of $D_0$ is $2g-2$.

We enumerate the points $p_i$ so that
 $\alpha_1\leq \alpha_2 \leq \dotso \leq \alpha_r$. 
The variety $\overline{X}$ has cyclic quotient singularities of type $A_{\alpha_1-1},\dotsc,A_{\alpha_r-1}$ along $\overline{X}_\infty:=\overline{X} \setminus X$. Let $\pi\colon S\to\overline{X}$ be the minimal normal crossing resolution of all singularities of $\overline{X}$. The preimage $\bE:=\widetilde{X}_\infty$ of $\overline{X}_\infty$ under $\pi$ consists of the strict transform $E$ of $\overline{X}_\infty$ and $r$ chains $\E{1}{i},\dotsc,\E{\alpha_i-1}{i}$, $i=1,\dotsc,r$, of rational curves of self-intersection number $-2$ which intersect according to the dual graph shown in Figure~\ref{Fig1}. The central curve $E$ is a curve of genus $g$ with self-intersection number $2g-2$.

\begin{figure}\centering
\psfrag{e11}{$\scriptstyle \E{1}{1}$}
\psfrag{e21}{$\scriptstyle \E{1}{2}$}
\psfrag{er1}{$\scriptstyle \E{1}{r}$}
\psfrag{er2}{$\scriptstyle \E{2}{r}$}
\psfrag{e1a1}{$\scriptstyle \E{\alpha_1-1}{1}$}
\psfrag{e2a2}{$\scriptstyle \E{\alpha_2-1}{2}$}
\psfrag{erar}{$\scriptstyle \E{\alpha_r-1}{r}$}
\psfrag{e}{$\scriptstyle E$}
\psfrag{g}{$\scriptscriptstyle 2g-2$}
%\psfrag{u}{$\scriptstyle E-u$}
%\psfrag{w}{$\scriptstyle u-w$}
\includegraphics[width=0.5\linewidth]{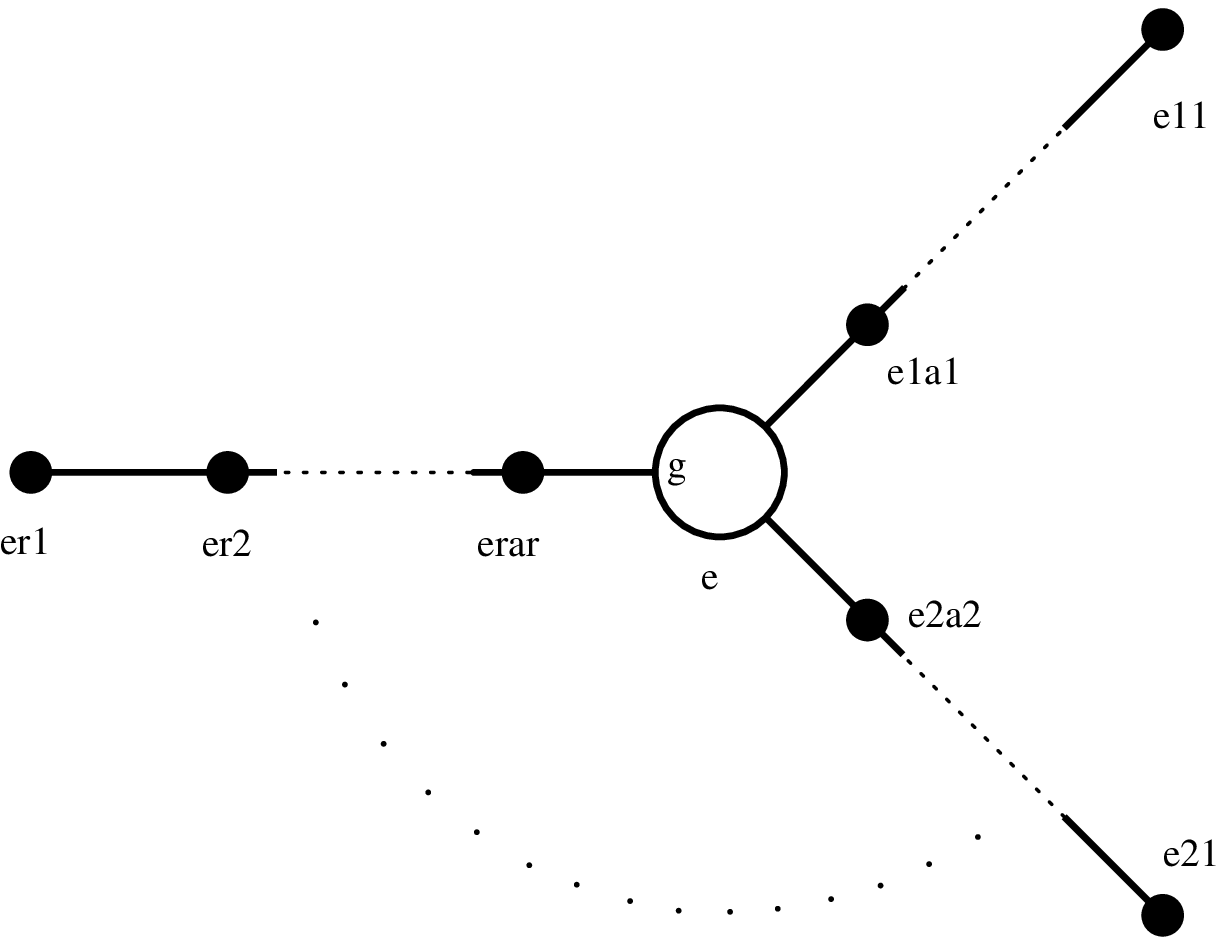}
\caption{Dual graph of $\bE=\widetilde{X}_\infty$}
\label{Fig1}
\end{figure}

We consider the \emph{Poincar\'{e} series} of the algebra $A$ 
\[ p_A(t)= \sum_{k=0}^\infty \dim(A_k) \ t^k . \]
In order to give a description for $p_A$, we need some definitions.

Let $V_-$ be the lattice generated by the irreducible components of $\widetilde{X}_\infty=\bE$ with bilinear form $\langle-,-\rangle$ given by the intersection numbers. Define two more lattices as orthogonal directs sums by $V_0 := V_- \oplus \ZZ u$ and $V_+:= V_- \oplus U$ where $\langle u,u\rangle=0$. For both $V_0$ and $V_+$, denote by $\psi_{u,E}$ the Eichler-Siegel transformation corresponding to $u$ and $E$, as introduced in Example \ref{ex:U}. Define isometries $\tau_0\in\text{O}(V_0)$ and 
$\tau_+\in\text{O}(V_+)$ by
\begin{align*}
 \tau_0 &:= s_{\E{1}{1}} \dotsm s_{\E{\alpha_1-1}{1}} \dotsm s_{\E{1}{r}} \dotsm s_{\E{\alpha_r-1}{r}} \psi_{u,E}, \\
 \tau_+ &:= s_{\E{1}{1}} \dotsm s_{\E{\alpha_1-1}{1}} \dotsm s_{\E{1}{r}} \dotsm s_{\E{\alpha_r-1}{r}} \psi_{u,E} s_{u-w}.
\end{align*}
Let $\Delta_0(t) = \det (1- \tau_0^{-1} t)$ and $\Delta_+(t) = \det (1- \tau_+^{-1} t)$ be the characteristic polynomials of $\tau_0$ and $\tau_+$ respectively, using a suitable normalization.

\begin{theorem} \label{thm:main}
For a Fuchsian singularity we have $ p_A = \dfrac{\Delta_+}{\Delta_0} $.
\end{theorem} 

\begin{remark}
In \cite{E03}, the following formula is proved: $p_A=\Delta_+/\psi_A$ where  $\psi_A(t)= (1-t)^{2-2g -r} \prod_{i=1}^r (1-t^{\alpha_i})$.
Note that Eichler-Siegel transformations are used. In view of the approach of this work, the denominator $\psi_A$ amounts to a factorisation of $\Delta_0$. In \cite[Remark 1]{E03} it is also observed that $\psi_A= \Delta_0$ if $g=0$.
\end{remark}

%%%%%%%%%%%%%%%%%%%%%%%%%%%
\section{The proof}
The proof of Theorem~\ref{thm:main} consists of two steps. First, we consider a general even lattice $V_-$ possessing a basis with at most one non-root. For such lattices, we develop the rational function $\Delta_+/\Delta_0$ into a formal power series. This is inspired by Lenzing's approach \cite{L1} but different both in details and in spirit; for the latter, see Remark \ref{lem:curve} below.
In the second step, we show that this power series coincides with the Poincar\'e series if we start with the lattice coming from a Fuchsian singularity. 

\subsection{Hilbert-Poincar\'e series for even lattices almost generated by roots}

Let $(V_-, \langle -,- \rangle)$ be an even lattice with basis $e_1,\dotsc,e_{n-1},e$ where $e_1,\dotsc,e_{n-1}$ are roots. Here $n \geq 1$ and $V_-=\ZZ e$ if $n=1$. Define $g\in\ZZ$ by $\langle e, e \rangle=2g-2$. Consider the lattices $V_0:=V_-\oplus\ZZ u$ and $V_+:=V_-\oplus U$ defined as before. Let $\tau_0 = s_{e_1} \dotsm s_{e_{n-1}} \psi_{u,e}$ and $\tau_+=\tau_0 s_{u-w}$. Note that $\tau_0$ can be seen as an isometry of $V_+$ or as an isometry of $V_0$. We write $\tau_0|_{V_0}$ if we mean the latter.

Define a Hilbert-Poincar\'e series corresponding to $V_0$ and $e$ as follows:
\[ P_{(V_0,e)}(t):= \sum_{k=0}^\infty \Big(1-g+\sum_{\ell=0}^{k-1} 
                   \langle e,\tau_0^\ell(e)\rangle\Big) \ t^k. \]

\begin{proposition}  \label{prop:LP}
We have $\det(1- \tau_+^{-1} t) \det(1- \tau_0|_{V_0}^{-1} t)^{-1} = P_{(V_0,e)}(t)+g+t$.
\end{proposition}

\begin{proof}
Borrowing an idea of \cite[\S18]{L1}, we use formal power series to invert $1- \tau_0^{-1} t$. Consider the linear operator on
 $V_+\pseriesl t\pseriesr :=\ZZ\pseriesl t\pseriesr\otimes V_+$ given by
\begin{align*}
 h &:= (1-\tau_+^{-1}  t)(1-\tau_0^{-1} t)^{-1}
     = (1- s_{u-w} \tau_0^{-1} t) \Big(\sum\nolimits_{k\geq0} \tau_0^{-k} t^k \Big) \\
   &\phantom{:}= 1 - \sum\nolimits_{k\geq1}\big\langle\,\cdot\,,\tau_0^k(u-w)\big\rangle t^k (u-w) ,
\end{align*}
where the last equation follows from an easy computation unravelling $s_{u-w}$. For all $v\in V_+$, we thus find 
$h(v)\in v+\ZZ\pseriesl t\pseriesr (u-w)$, hence $h(u-w) = \det(h) (u-w)$. Furthermore invoking
$\det(1- \tau_0^{-1} t) = (1-t) \det(1- \tau_0|_{V_0}^{-1} t)$, we see that
\begin{align*}
    \det(1- \tau_+^{-1} t) \det(1 - \tau_0|_{V_0}^{-1} t)^{-1}
 &= (1-t) \det(h) \\
 & = (1-t) \Big(1-\sum\nolimits_{k\geq1}\langle u-w,\tau_0^k(u-w)\rangle t^k \Big) \\
 &= 1+t - \sum\nolimits_{k\geq1} \big\langle u-w, (\tau_0^k-\tau_0^{k-1})(u-w) \big\rangle t^k .
\end{align*}
Plugging in the definition of $\psi_{u,e}$ now yields
\[ \langle u-w, (\tau_0^k-\tau_0^{k-1})(u-w) \rangle =
   \tfrac{1}{2} \langle e,e \rangle - \langle u-w, \tau_0^k(e) \rangle  \]
and $\langle \tau_0^k(e),u-w \rangle = \sum_{\ell=0}^{k-1} \langle e,\tau_0^\ell(e)\rangle$.
Putting these pieces together gives the claim.
\end{proof}

\subsection{Poincar\'e series of Fuchsian singularities}
Let now $V_0$ be the lattice introduced in Section \ref{sec:result}, generated by $u$ and the components of $\bE$. Theorem \ref{thm:main} follows at once from Propositions \ref{prop:LP} and \ref{prop:Poincare}.

\begin{proposition} \label{prop:Poincare}
The Poincar\'e series of a Fuchsian singularity is $P_{(V_0, E)}(t) +g+ t$.
\end{proposition}

\begin{proof}
%Let $\e:=E$, $\be{j}{i}:=\E{j}{i}$.
The element $u$ spans the radical $\rad(V_0)$ of the lattice $V_0$. 
Let $\overline{V}_0= V_0/\rad(V_0)$. For an automorphism $\sigma$ of $V_0$, we
denote by the same letter the induced automorphism
 $\sigma\colon \overline{V}_0 \to \overline{V}_0$.
By more abuse of notation, we will denote elements of $V_0$ and their classes in 
$\overline{V}_0$ by the same letter. In order to compute the series $P_{(V_0,\e)}(t)$, 
it suffices to consider the automorphism $\tau_0$ of $\overline{V}_0$. 
Now one can easily see that, on the quotient space, $\psi_{u,\e}=\text{id}_{\overline{V}_0}$. 
Therefore we have, again on the quotient $\overline{V}_0$,
\[ \tau_0 = \tau_1 \dotsm \tau_r \mbox{ where } 
   \tau_i:= s_{\be{1}{i}} \dotsm s_{\be{\alpha_i-1}{i}} \mbox{ for } \ i=1, \dotsc , r. \]
For $\tau_i$ we have
 $\tau_i(\e)=\e+ \sum_{j=1}^{\alpha_i-1} \be{j}{i}$, 
 $\tau_i(\be{j}{i})=\be{j+1}{i}$ 
for $j=1, \dotsc , \alpha_i-2$, 
$\tau_i(\be{\alpha_i-1}{i})= - \sum_{j=1}^{\alpha_i-1} \be{j}{i}$, and $\tau_i$ 
is the identity on all other basis elements of $\overline{V}_0$. Therefore,
if $\alpha_i > 2$ (otherwise $\tau_i^2(\e)=\e$)
\[ \tau_i(\e)=\e+ \sum_{j=1}^{\alpha_i-1} \be{j}{i}, \quad
   \tau_i^2(\e)= \e+ \sum_{j=2}^{\alpha_i-1} \be{j}{i}, \quad 
   \dotsc, \quad
   \tau_i^{\alpha_i}(\e)=\e. \]
This implies, for $k>0$,
\begin{equation}  \label{FuchsDiv}
   \big\langle \e, \sum_{\ell=0}^{k-1} \tau_0^\ell \e \big\rangle 
 = \big\langle \e, k\e 
   + \sum_{i=1}^r \big[ k \tfrac{\alpha_i - 1}{\alpha_i} \big]  \be{\alpha_i -1}{i} \big\rangle
 = \langle E, E^{(k)} \rangle 
 = \deg D^{(k)}.
\end{equation}
Here 
$$ E^{(k)} := kE + \sum_{1\leq i\leq r} \sum_{1\leq j<\alpha_i} \big[ \tfrac{kj}{\alpha_i} \big] \E{j}{i}$$
is the total transform of the Weil divisor $k\overline{X}_\infty$ under $\pi:S\to\overline X$ \cite[6.4]{Pinkham78}. We would like to emphasize that the last equality in (\ref{FuchsDiv}) is not just numerical but stems from an isomorphism $\calO_S(E^{(k)})|_E\cong\calO_E(D^{(k)})$, see \cite[\S6]{Pinkham78}.

We have $\deg D^{(k)} > 2g-2$ for any $k>1$. This is obvious if $g>1$; in the remaining cases it follows from the relation $\sum_{i=1}^r \frac{1}{\alpha_i} < r+2g-2$ of \cite[(1.7)]{Looijenga84}. As the degree of $D^{(k)}$ is large enough, we get $\dim L(D^{(k)}) = 1-g+\deg D^{(k)}$ by Riemann-Roch. Finally, the two series coincide for $t=0$ as well since $\dim L(D^{(1)}) = \dim L(D_0)=g$.
\end{proof}

%%%%%%%%%%%%%%%%%%%%%%%%%%%
\section{The interpretation: Lifting Coxeter elements to functors}

The lattice $V_-$ is constructed from the geometry, but its extensions $V_0$ and $V_+$ are purely algebraic, as is the Eichler-Siegel transformation. Similar to our approach in \cite{EP}, we are interested in a geometric interpretation of these invariants. Roughly speaking, this is achieved by constructing, in a natural manner out of the geometry, triangulated categories whose numerical K-groups yield these lattices. In turn, the Coxeter elements can be lifted to invertible functors. The idea of Coxeter functors is certainly not new, see Remark \ref{lem:curve} below.
For principal reasons, our method only applies to a certain (large) class of Fuchsian singularities, those which are negatively smoothable.

\subsection{Negatively smoothable Fuchsian singularities}
Let $\pi\colon S\to\overline{X}$ be the minimal normal crossing resolution of $\overline{X}$. It contains
the total transform of the curve at infinity, $\bE=\widetilde{X}_\infty:=\pi^{-1}(\overline{X}_\infty)$.
This is a configuration of curves, all but one of which are smooth, rational $(-2)$-curves.

We assume that $X$ is \emph{negatively smoothable} (for the definition see
\cite{Pinkham78}). This implies that there is a deformation of $\overline{X}$ with the following properties: all members of the family are partial resolutions of $\overline{X}$ (namely, the singularities at infinity are resolved); each member contains the curve configuration $\bE$; the generic 
fibre is smooth. By \cite[Proposition 6.13]{Pinkham78}, the generic fibre is a smooth K3 surface.

For example, if $(X,x)$ is an isolated hypersurface or complete intersection singularity, then it is negatively smoothable. On the other hand, since the rank of the N\'{e}ron-Severi group of a K3 surface is at most 20, a necessary condition for negative smoothability is $\sum_i\alpha_i\leq19+r$.

\medskip

\noindent
\emph{In the sequel, we assume that $(X,x)$ is negatively smoothable. Let $\pares$ be a generic fibre, it is a smooth K3 surface containing the configuration $\bE$.}

\subsection{Lattices from the K-groups}

Let $\Coh(\pares)$ be the abelian category of coherent sheaves on $\pares$ and $K(\pares)$ its Grothendieck K-group. The Euler pairing on $K(\pares)$, defined as $\chi(A,B)=\sum_i(-1)^i\dim\Ext^i_\pares(A,B)$ for coherent sheaves $A$ and $B$ on $\pares$, is symmetric by Serre duality and the fact that $\pares$ is a K3 surface. Note that  $\chi(A):=\chi(\calO_\pares,A)$ is the Euler characteristic of a sheaf $A\in\Coh(\pares)$. We equip the K-group (and all groups derived from it) with the \emph{negative} Euler pairing.
Let $N(\pares)$ be the numerical K-group which is obtained from $K(\pares)$ by dividing out the radical of the Euler form.

Denote by $\Coh_\bE(\pares)$ the abelian subcategory of $\Coh(\pares)$ consisting of sheaves whose support is contained in $\bE$ and let $K_\bE(\pares)$ be its K-group. Let $N_\bE(\pares)$ be the image of $K_\bE(\pares)$ under $K(\pares)\to N(\pares)$. 

Using the notation of Figure~\ref{Fig1} and choosing a point $p\in E$, we will consider the following sheaves supported on $\bE$
\[ \F{j}{i} := \calO_{\E{j}{i}}(-1), \qquad
    F\in\Pic^{g-1}(E) \text{ with } H^0(F)=0, \qquad
    \widetilde F := F(p) . \]
So $F$ is a line bundle of degree $g-1$ supported on $E$. The condition $H^0(F)=0$ implies $H^1(F)=0$ by Riemann-Roch. Line bundles without global sections make up the complement of the theta-divisor in $\Pic^{g-1}(E)$. In particular, such line bundles are not  unique except for $F=\calO_E(-1)$ if $g=0$.

The classes in $N_\bE(\pares)$ of these sheaves form a basis. This is well-known, but see Subsection \ref{sub:cohomology} for details.
In this setting, we define three sublattices of $N(\pares)$.
Recall that we have equipped the numerical K-groups with the negative Euler pairing.
\begin{align*}
 V'_0 &:= N_\bE(\pares) && \text{with } u':=[F]-[\widetilde F]=-[k(p)], \\
 V'_+ &:= N_\bE(\pares)\oplus\ZZ[\calO_\pares] && \text{with } w':=[\calO_\pares]+u'=[\calI_p], \\
 V'_- &:= U'^\perp = N_\bE(\pares)\cap[\calO_\pares]^\perp && \text{with } U':=\ZZ u'+\ZZ w'.
\end{align*}
Note that the class $-u'$ is represented by the skyscraper sheaf $k(p)$ of $p$ and that $w'$ is represented by the ideal sheaf $\calI_p$ of $p$. In the derived category, the class $u'$ is thus given by the shift $k(p)[1]$. It is worth pointing out that the structure sheaves of different points are in general not identified in $K(\pares)$ but that they all represent the same class in $N(\pares)$. Similarly, all choices for $F$ lead to the same class in $N(\pares)$. The alternative description of $V'_-$ follows from the next lemma which also shows that these lattices are indeed isometric to the ones used before.

\begin{lemma} \label{lem:Niso}
The map $\eta_-\colon V_- \to V'_-$ defined by $\E{j}{i}\mapsto\F{j}{i}$ and
$E\mapsto F$ is an isometry. The extensions $\eta_0\colon V_0\to V'_0$ and $\eta_+\colon V_+\to V'_+$
of $\eta_-$ mapping $u\mapsto u'$ and $w\mapsto w'$ are isometries.
\end{lemma}

\begin{proof}
First, $U'$ is a unimodular hyperbolic plane: $\Ext^0(k(p),k(p))=\CC$ implies via Serre duality $\Ext^2(k(p),k(p))=\CC$, and $\Ext^1(k(p),k(p))=\CC^2$ as $\pares$ is a smooth surface. This shows $\chi(u',u')=0$. From $w'=[\calO_\pares]+u'$ we get $-\chi(u',w')=1$. Finally, $\chi(w',w')=0$ follows from $\chi(\calO_\pares)=2$ for the K3 surface $\pares$.

We now show that $\eta_-$ is well-defined, i.e.\ takes values in $U'^\perp$. Let $C$ be any irreducible component of $\bE$. The short exact sequence $0\to\calO_\pares(-C)\to\calO_\pares\to\calO_C\to0$ yields $\chi(\calO_C,u')=0$. Next, $[\calO_C(D)]=[\calO_C]-\deg(D)u'$ for any $D\in\Div(C)$ and hence $\chi(\calO_C(D),u')=0$. Furthermore, $\chi(\calO_C)=1-g_C$, so that $\deg(D)=g_C-1$ implies $\chi(\calO_C(D))=0$ and so $\chi(\calO_C(D),w')=0$.

The map $\eta_-$ is a bijection as it is injective and $V'_-$ and $V_-$ are free abelian groups of the same rank. It only remains to show that $\psi_-$\ respects the pairings. For any two irreducible curves $C$ and $C'$ on $\pares$, their intersection number can be computed as $C.C'=-\chi(\calO_C,\calO_{C'})$. 
This is immediate if $C$ and $C'$ are transversal, for then the only non-vanishing summand in the Euler pairing is $\dim\Ext^1(\calO_C,\calO_{C'})$, which is the number of intersection points. In case $C'=C$, the short exact sequence from above yields $-\chi(\calO_C,\calO_C)=\chi(\calO_\pares(C)|_C)-\chi(\calO_C)=\text{deg}(\calO_C(C))$, the last term being the self-intersection number by definition.
To conclude, just observe that the quantity $\chi(\calO_C(D),\calO_{C'}(D'))$ 
is not affected by the choice of $D\in\Div(C)$, $D'\in\Div(C')$, as above.
\end{proof}

\begin{remark} As in \cite{EP}, we can consider the group $N_\bE(S)$ where $\pi\colon S\to\overline{X}$ is the minimal normal crossing resolution of the compactification $\overline{X}$. Then the lattice $V_0$ is also isometric to this group endowed with the negative Euler pairing. In particular, the proof of Proposition \ref{prop:Poincare} was carried out with this realisation of $V_0$, which does not rely on an assumption of negative smoothability. However, it is not possible to extend the isometry $V_0\isom N_\bE(S)$ to an isometry of $V_+$ with a sublattice of $N(S)$.
\end{remark}

\subsection{Lifting Coxeter elements to functors}

Denote by $\DDD^b(\pares)$ the bounded derived category of coherent sheaves on $\pares$. Clearly, a triangle autoequivalence $\varphi\colon \DDD^b(\pares)\isom\DDD^b(\pares)$ descends to isometries $\varphi^K\in\text{O}(K(\pares))$ and $\varphi^N\in\text{O}(N(\pares))$ between (numerical) K-groups. We employ two types of geometrically defined functors in order to lift $\tau_0\in\text{O}(V_0)$ and $\tau_+\in\text{O}(V_+)$ to autoequivalences of $\DDD^b(\pares)$.

\paragraph{Spherical twists.} A coherent sheaf $G$ on a K3 surface is \emph{spherical} if $\Hom(G,G)=\CC$ and $\Ext^1(G,G)=0$. For such a sheaf, the functor $\TTT_G$ defined by distinguished triangles
 $\Hom^\bullet(G,A)\otimes G\to A \to\TTT_G(A)$ for any $A\in \DDD^b(\pares)$, is an autoequivalence of
$\DDD^b(\pares)$. (For a proof and the correct definition of spherical in the general context, see
\cite[\S8.1]{Huy}.) It is easy to see that $\TTT_G|_{G^\perp}=\text{id}$ where $G^\perp=\{A\in\DDD^b(\pares)\mid\Hom^\bullet(G,A)=0\}$ and that $\TTT_G(G)\cong G[-1]$.
Hence, the spherical twist induces the reflection $\TTT_G^K=s_{[G]}$ where $[G]\in K(\pares)$ is by sphericality a root (for the negative Euler pairing); analogously $\TTT_G^N=s_{[G]}$.

Note that $\calO_\pares$ is a spherical sheaf. If $i:C\hookrightarrow\pares$ is the embedding of a smooth, rational $(-2)$-curve, then $i_*\calO_C(n)$, abusively denoted by $\calO_C(n)$, is spherical for any $n\in\ZZ$, since $C$ is rigid.

\paragraph{Line bundle twists.} A line bundle $L\in\text{Pic}(\pares)$ gives rise to the auto\-equivalence $\MMM_L\colon \DDD^b(\pares)\to \DDD^b(\pares)$, $A\mapsto L\otimes A$. 
% The induced map on K-groups is just multiplication by $[L]$. 
Decompose $[L]=w+\ell+du$ in $N(\pares)$, where $d\in\ZZ$ and $\ell\in N(\pares)$ with $\chi(w,\ell)=\chi(u,\ell)=0$. We claim that $\MMM_L^N=\psi_{u,\ell}=m_\ell$ is an Eichler-Siegel transformation.

First, as $\pares$ is a K3 surface, we have $\chi(L,L)=\chi(\calO_\pares)=2$. Together with
$\chi(L,L)=\chi(w+\ell+du,w+\ell+du)=\chi(\ell,\ell)-2d$, this shows $d=\frac{1}{2}\chi(\ell,\ell)-1$.
% Likewise, $\chi(L)=\chi(w-u,w+\ell+du)=-d+1$ and $d=1-\chi(L)$.
The claim follows from $\MMM_L^N(u)=u$, $\MMM_L^N(w)=w+\ell+\frac{1}{2}\chi(\ell,\ell)u$ and $\MMM_L^N(v)=v+\chi(\ell,v)u$ for $v\in U^\perp$. The first two equations are obvious. For the third, without loss of generality assume $v=[D]$ with $D\in\Pic(C)$, $\deg(D)=g_C-1$ and write $L=\calO_\pares(A-A')$ with ample, effective divisors $A$ and $A'$, both meeting $C$ transversally. The sequences $0\to L|_C\to\calO_\pares(A)_C\to\calO_{A'\cap C}\to0$ and $0\to\calO_C\to\calO_\pares(A)|_C\to\calO_{A\cap C}\to0$ are exact by the transversality assumptions, leading to $[L\otimes\calO_C]=[\calO_C]+ku$ and $[L^\vee]=w-\ell+d'u$ for some $k,d'\in\ZZ$. Hence, $\MMM_L^N(v)=v+k'u$ and the coefficient $k'$ is readily computed as 
 $-k' = \chi(w,v+k'u) = \chi(\calO_\pares,L\otimes D)
      = \chi(L^\vee,D)= \chi(w-\ell+d'u,v) = -\chi(\ell,v)$.

As an example, $[\calO_\pares(E)]=w+[F]+du$ for the smooth curve $E$ of genus $g$ in $\pares$. Hence
$\MMM_{\calO_\pares(E)}^N=m_{E}$, identifying $\eta_-(E)=[F]$.

\bigskip
\noindent
Consider the full triangulated subcategory $\DDD^b_\bE(\pares)$ consisting of complexes whose support is contained in $\bE$ (in other words, which are exact off $\bE$). 
% By the next lemma, this category is actually the derived category of $\Coh_\bE(\pares)$. 
The following full triangulated subcategories of $\DDD^b(\pares)$ will be used:
\begin{align*}
 \DDD_-    &:= \DDD^b_\bE(\pares) \cap \calO_\pares^\perp, \\
 \DDD_0 \, &:= \DDD^b_\bE(\pares) , \\ 
 \DDD_+    &:= \langle \DDD^b_\bE, \calO_\pares\rangle,
\end{align*}
i.e.\ $\DDD_+$ is the smallest full triangulated subcategory of $\DDD^b(\pares)$ containing $\DDD_0$ and $\calO_\pares$ (this decomposition is not semiorthogonal). Note that $\E{j}{i}\in\DDD_-$ and $F\in\DDD_-$ by construction.

\begin{lemma} We have:

\smallskip
\noindent
\begin{tabular}{rl}
\upshape{(i)}  & $N(\DDD_-)=V'_-$ and $N(\DDD_+)=V'_+$. \\
\upshape{(ii)} & $V'_0$ is the image of $K(\DDD_0)\hookrightarrow K(\DDD_+)\to N(\DDD_+)$.
%\upshape{(i)}   & The canonical functor $\DDD^b(\Coh_\bE(\pares))\to\DDD^b_\bE(\pares)$ is an equivalence. \\
\end{tabular}
\end{lemma}

\begin{proof}
(i) is obvious from Lemma \ref{lem:Niso} and the definitions of $V'_-$, $V'_+$ and $\DDD_-$, $\DDD_+$, respectively. For (ii), just note that $F$ and $\widetilde F$ yield the class of a point in the numerical K-group. Also note that the lattices $N(\DDD_0)$ and $V'_0$ are not isomorphic since $[k(p)]\in\rad(K(\DDD_0))$ by Lemma \ref{lem:Niso} for the class of (the skyscraper sheaf of) a point. 
%
%(iii) The category $\Coh_\bE(\pares)$ inherits local Ext sheaves from $\Coh(\pares)$. Hence, the global Ext groups coincide by the standard spectral sequence. Thus the functor is fully faithful on $\Coh_\bE(\pares)$. By an induction on the length of triangles, it is fully faithful altogether. On essential surjectivity: we have to show that any complex $A^\bullet\in\DDD^b(\pares)$, exact off $\bE$, is isomorphic to a complex comprised of sheaves set-theoretically supported on $\bE$. For this, consider the thickened neighbourhoods $\bE^k$ defined by powers of the ideal of $\bE$. By assumption, each cohomology sheaf $h^i(A^\bullet)$ is supported scheme-theoretically on some $\bE^k$. As $A^\bullet$ is bounded, there is a $k\gg0$ such that $A^\bullet\to A^\bullet|_{\bE^k}\in\DDD^b(\Coh_\bE(\pares))$ is a quasi-isomorphism.
\end{proof}

We proceed to define the autoequivalences of $\DDD^b(\pares)$ which lift the Coxeter elements:
\begin{align*}
 \varphi_0 \, &:= \TTT_{F_1^1} \dotsm \TTT_{F_{\alpha_1-1}^1} \dotsm
                  \TTT_{F_1^r} \dotsm \TTT_{F_{\alpha_r-1}^r} \MMM_{\calO_\pares(E)}, \\
 \varphi_+    &:= \TTT_{F_1^1} \dotsm \TTT_{F_{\alpha_1-1}^1} \dotsm
                  \TTT_{F_1^r} \dotsm \TTT_{F_{\alpha_r-1}^r} \MMM_{\calO_\pares(E)} 
                  \TTT_{\calO_\pares}
                = \varphi_0 \TTT_{\calO_\pares} .
\end{align*}

\begin{theorem}
The autoequivalences $\varphi_0$ and $\varphi_+$ restrict to autoequivalences of $\DDD_0$ and $\DDD_+$, respectively, and $\varphi_0^N=\tau_0$ and $\varphi_+^N=\tau_+$.
\end{theorem}

\begin{proof}
Most of the assertions in the theorem were proven in the preceding discussion. Note that $w-u=[\calO_\pares]$, so $\TTT_{\calO_\pares}=s_{w-u}=s_{u-w}$, as desired. What remains to be shown is $\varphi_0(\DDD_0)=\DDD_0$ and $\varphi_+(\DDD_+)=\DDD_+$.

For an arbitrary line bundle $L\in\text{Pic}(\pares)$, the autoequivalence $\MMM_L$ of $\DDD^b(\pares)$ respects supports. Hence, $\MMM_L$ maps $\DDD_0$ into $\DDD_0$. As $\calO_\pares(E)\in\DDD_+$, the functor $\MMM_{\calO_\pares(E)}$ maps $\DDD_+$ into $\DDD_+$.

Turning to the spherical twist functors, the following fact proves the claim: For a full triangulated subcategory $\mathcal{T}\subset\DDD^b(\pares)$ and a spherical object $G\in\mathcal{T}$, the twist $\TTT_G$ restricts to an autoequivalence of $\mathcal{T}$, as follows at once from the triangles defining $\TTT_G$.
\end{proof}

\begin{remark}
There are many other categories that can be used here. For example, instead of $\DDD_+$ one could as well take the triangulated category generated by the structure sheaves of the surface, the irreducible components of $\bE$ and of a point on $E$. The categories we employ are natural --- they do not depend on additional choices. However, note that the Coxeter functors depend on the order of the spherical twists.
\end{remark}

\begin{remark}
The case $g=0$ has already been treated in \cite{EP}. While making use of the same triangulated categories $\DDD_0$ and $\DDD_+$, different functors were presented as lifts of $\tau_0$ and $\tau_+$: since $g=0$ means that the central curve $E$ is rational, the additional spherical objects $\calO_E$ and $\calO_E(-1)$ can be used to take $\TTT_{\calO_E(-1)} \TTT_{\calO_E}$ in place of $\MMM_{\calO_\pares(E)}$.

We remark that $\DDD_0$ and $\DDD_+$ are generated by spherical objects if and only if $g=0$, and then the autoequivalences $\TTT_{\calO_E(-1)} \TTT_{\calO_E}$ and $\MMM_{\calO_\pares(E)}$ are genuinely different. For example, $\MMM_{\calO_\pares(E)}(\calO_E)$ is a sheaf but $\TTT_{\calO_E(-1)} \TTT_{\calO_E}(\calO_E)$ has nonzero cohomology in two degrees.
\end{remark}

\begin{remark} \label{lem:curve}
There should be a curve picture of the situation, analogous to the one developed in \cite{KST} for certain hypersurface singularities. More precisely, there should be a lift of the Coxeter element to the graded triangulated category of singularities, $\Dsgr(R)$. In the case when $g=0$ and $(X,x)$ is a hypersurface singularity, it is proved in \cite{KST} that $\Dsgr(R)$ is generated by a collection of exceptional objects, turning it into the derived category of a quiver with relations. Nevertheless, if $g>0$, one cannot expect to have a full exceptional collection but only a differential graded algebra as model. The approach to generating series pursued in \cite{L1} is modelled on the curve case, using non-symmetric forms and roots of length 1. 

In \cite{BGP}, the Coxeter element (of a root lattice) is lifted to an endofunctor of the category of representations of the (oriented) quiver. This functor is not invertible, which seems to be related to the fact that the category used is abelian and not triangulated. 
\end{remark}

\subsection{Cohomology instead of K-group} \label{sub:cohomology}

We close by pointing out that the lattices $V_0$ and $V_+$ can also be obtained from the numerical Chow group or from cohomology. In fact, these two invariants seem to be used more often than the numerical K-group, so we briefly explain the differences. Hitherto, we have opted to work with the (numerical) K-groups exclusively because these are truly intrinsic invariants of the triangulated categories.

The Chern character defines a map $\chern\colon K(\pares)\to CH^*(\pares)\otimes\QQ$ and, by the Riemann-Roch theorem, an isomorphism 
$K(\pares)\otimes\QQ\isom CH^*(\pares)\otimes\QQ$, see \cite[Corollary 18.3.2]{Fulton}.
As $\pares$ is a surface with even intersection pairing, the Chern map is already defined without denominators. Next, there is the cycle map $CH^*(\pares)\to H^*(\pares)$ to singular cohomology with integral coefficients: its image is the algebraic part of cohomology.

% It is well-known that $H^*(\pares)$, equipped with the cup product, is isomorphic as a lattice to $-2E_8+3U$. 

As $\pares$ is a smooth, projective surface, $CH^1(\pares)_{\text{num}}$ is isomorphic to the N\'eron-Severi group of $\pares$ and $CH^2(\pares)_{\text{num}}$ is free of rank one, spanned by the class of a point. We find that the Chern map induces an isomorphism 
 $N(\pares)\isom CH^*(\pares)_\text{num}$
which, however, is not an isometry. The cycle map does respect the pairings and yields an isometry
$CH^*(\pares)_{\text{num}}\isom H^*(\pares)_{\text{alg}}$.

Matters can be improved by taking the Mukai vector
$v(\cdot):=\chern(\cdot)\sqrt{\text{td}_\pares}$
instead of the Chern character (where $\text{td}_\pares$ is the Todd class of the surface), and by modifying the pairings on Chow ring and cohomology: invert the sign of the unimodular hyperbolic plane 
spanned by fundamental class and point. Denoting this new pairing by $\langle-,-\rangle$, the Grothendieck-Riemann-Roch theorem gives 
$\chi(A,B)=-\langle v(A),v(B)\rangle$
for all coherent sheaves on $A$ and $B$. (As $\text{td}_\pares=[\pares]-2u$, where $[\pares]$ is the class of the surface and $-u$ is the class of a point, we have 
 $v(A)=\chern(A) - \text{rk}(A)u$ for any $A\in\Coh(\pares)$.)
See \cite[\S10]{Huy} for details.

Consequently, we arrive at a chain of lattice isomorphisms
\[ N(\pares) \xrightarrow{v} CH^*(\pares)_{\text{num}} 
             \to H^*(\pares)_{\text{alg}} .\]
Note that $v(\F{i}{j})=\chern(\F{i}{j})=[\E{i}{j}]$ and $v(F)=\chern(F)=[E]$ 
as cycles in the numerical Chow group or cohomology. The class of a point 
is given by $pt=v(\widetilde F)-v(F)=-u$.

An autoequivalence $\varphi\in\text{Aut}(\DDD^b(\pares))$ induces isomorphisms
$\varphi^{CH}$ and $\varphi^H$ of the Chow ring and cohomology, respectively.
In contrast to $\varphi^K$ and $\varphi^N$, this is not tautological but relies
on Orlov's existence theorem for Fourier-Mukai kernels on smooth, projective
varieties; see \cite[\S5]{Huy} for details. The maps $\varphi^{CH}$ and
$\varphi^H$ are isometries for the Mukai pairings indicated above. One can easily
check that for a spherical sheaf $G$ on $\pares$, $\TTT_G^H=s_{v(G)}$ is the
reflection along its Mukai vector. Given a line bundle $L$, $\MMM_L^H=m_{c_1(L)}$
is the Eichler-Siegel transformation for the first Chern class of $L$; this is 
also multiplication (using the cup product) with the Chern character of $L$.

%The image $\chern(N_\bE(\pares))$ under the Chern isomorphism is 
%generated by the irreducible components of $\bE$ and the class of a 
%point, $pt$. It is easy to see that $pt$ is in the radical of 
%$K_\bE(\pares)$ (\cite[Lemma~4]{EP}). But $pt$ is not in the radical 
%of $K(\pares)$ by virtue of $\chi(\calO_\pares,k(p))=1$ and hence is 
%a nonzero class in $N_\bE(\pares)$.

\bigskip
\noindent Wolfgang Ebeling, Leibniz Universit\"{a}t Hannover, Institut f\"{u}r Algebraische Geometrie,
Postfach 6009, D-30060 Hannover, Germany \\
E-mail: ebeling@math.uni-hannover.de

\bigskip
\noindent David Ploog, Department of Mathematics, University of Toronto,
Toronto, Ontario, Canada M5S 2E4 \\
E-mail: ploog@math.uni-hannover.de


\begin{thebibliography}{Wag1}

\linespread{0.75}
\setlength{\parskip}{0.0ex}
\small{

\bibitem[BGP]{BGP} I.~N.~Bernstein, I.~M.~Gelfand, V.~A.~Ponomarev:
Coxeter functors and Gabriel's theorem.
% Uspehi Mat.\ Nauk {\bf 28}, 19--33 (1973) (Engl.\ translation in
Russian Math.\ Surveys {\bf 28},  17--32 (1973).

\bibitem[Dol]{Dolgachev75} I.~V.~Dolgachev: Automorphic forms and weighted homogeneous
singularities. 
% Funkt.\ Anal.\ Jego Prilozh. {\bf 9}:2, 67--68 (1975) (Engl.\ translation in 
Funct.\ Anal.\ Appl.\ {\bf 9}, 149--151 (1975).

\bibitem[Eb]{E03} W.~Ebeling: The Poincar\'e series of some special
quasihomogeneous surface singularities. Publ.\ RIMS, Kyoto Univ.\ {\bf 39},
393--413 (2003).

\bibitem[EP]{EP} W.~Ebeling, D.~Ploog: McKay correspondence for the Poincar\'e 
series of Kleinian and Fuchsian singularities. arXiv: math.AG/0809.2738.

\bibitem[Ei]{Ei} M.~Eichler: Quadratische Formen und orthogonale Gruppen. Zweite Auf\-lage. Springer-Verlag, Berlin Heidelberg he New York, 1974.

\bibitem[Ful]{Fulton} W.~Fulton: Intersection theory. Springer-Verlag, Berlin etc., 1984.

\bibitem[Huy]{Huy} D.~Huybrechts: Fourier-Mukai transforms in algebraic geometry. 
Oxford Mathematical Monographs, 2006.

\bibitem[KST]{KST} H.~Kajiura, K.~Saito, A.~Takahashi: Triangulated categories of matrix
factorizations for regular systems of weights with $\varepsilon=-1$. arXiv: math.AG/0708.0210.

\bibitem[Len]{L1} H.~Lenzing: Coxeter transformations associated with finite-dimensional
algebras. In: Computational methods for representations of groups and
algebras (Essen, 1997; P.~Dr\"axler, G.~O.~Michler and C.~M.~Ringel, eds.), 
Progr. Math. Vol.~{\bf 173}, Birkh\"auser, Basel, 1999, pp.~287--308.

\bibitem[Lo]{Looijenga84} E.~Looijenga: The smoothing components of a triangle
singularity. II. Math.\ Ann.\ {\bf 269}, 357--387 (1984).

\bibitem[P1]{Pinkham77a} H.~Pinkham: Normal surface singularities
with $\CC^\ast$ action. Math.\ Ann.\ {\bf 227}, 183--193 (1977).

\bibitem[P2]{Pinkham78} H.~Pinkham: Deformations of normal surface singularities
with $\CC^\ast$ action. Math. Ann. {\bf 232}, 65--84 (1978).

}
\end{thebibliography}
\end{document}